\documentclass{article}
\usepackage{amssymb}
\usepackage{amsmath, amsthm}
\usepackage{graphicx}
\usepackage{color}
\newtheorem{theorem}{Theorem}[section]
\newtheorem{lemma}{Lemma}[section]
\newtheorem{corollary}{Corollary}[section]
\newtheorem{definition}{Definition}[section]
\newtheorem{proposition}{Proposition}[section]
\newtheorem{remark}{Remark}[section]

\begin{document}
\title
{The Multidimensional Damped Wave Equation:  
Maximal Weak Solutions for Nonlinear Forcing via Semigroups and
Approximation}
\author{Joseph W. Jerome\footnotemark[1]}
\date{}
\maketitle
\pagestyle{myheadings}
\markright{Multidimensional damped wave equation}
\medskip

\vfill
{\parindent 0cm}

{\bf 2010 AMS classification numbers:} 
35L05, 35L20, 47D03 

\bigskip
{\bf Key words:} 
Nonlinear damped wave equation, operator forcing, semigroup,  
weak solution, telegraph equation

\fnsymbol{footnote}
\footnotetext[1]{Department of Mathematics, Northwestern University,
                Evanston, IL 60208. \\In residence:
George Washington University,
Washington, D.C. 20052}

\begin{abstract}
The damped nonlinear wave equation, also known as the nonlinear 
telegraph equation,
is studied within the
framework of semigroups and eigenfunction approximation. 
The linear semigroup assumes a central role: it is bounded 
on the domain of its generator 
for all time $t \geq 0$.
This permits eigenfunction approximation within the semigroup framework,
as a tool for the study of weak
solutions.
The semigroup convolution formula, known to be rigorous on the generator
domain, is extended to the interpretation of
weak solution on an arbitrary time interval. 
A separate approximation theory can be developed
by using the invariance of the semigroup on eigenspaces 
of the Laplacian as the system evolves.  

For (locally) bounded continuous
$L^{2}$ forcing, there is a natural derivation of a maximal solution,
which can logically include a constraint on the solution as well.
Operator forcing allows for
the incorporation of concurrent physical processes. 
A significant feature of the proof in the nonlinear case is verification
of successive approximation without standard fixed point analysis. 
\end{abstract}
\section{Introduction.}
\label{Introduction}
The damped nonlinear wave equation, also known as the nonlinear  
telegraph equation,
has assumed renewed importance in recent years
(cf.\ \cite{Ber,El-Azab,Jang,LW2,Mawhin,Mawhin2}).
Reference \cite{LW2} is novel
because of its
connection to MEMS modeling. Operator forcing, induced by a concurrent
physical process,  
and constraints come into
play in \cite{LW2} 
(see Table 1, p.\ 472, $\gamma > 0, \beta = 0$, for an
unresolved case within the framework we are studying). 

In this article, we consider the following 
initial/boundary-value problem.
The spatial domain $\Omega \subset {\mathbb R}^{N}$ represents a bounded
convex domain or a bounded $C^{2}$ domain. For these domains, $\nabla^{2}$ 
represents an algebraic and topological isomorphism of  
$H^{2}(\Omega) \cap H^{1}_{0}(\Omega)$ onto $L^{2}(\Omega)$
(\cite{A,GT}).
The system considered is the following, 
where subscripts represent
(partial) derivatives with respect to $t$. 
\begin{eqnarray}
u_{tt} &=& - \nu u_{t} + \kappa \nabla^{2} u + {\mathcal F}(u), 
\nonumber \\
u({\bf x},t)  &=& 0, \; {\bf x} \in \partial \Omega, t> 0, \nonumber \\ 
(u({\bf x}, 0), u_{t}({\bf x}, 0)) &=& (u_{0}({\bf x}), v_{0}({\bf x})), 
\; {\bf x} \in \Omega. 
\label{cibvp}
\end{eqnarray}
Here, $\nu$ and $\kappa$ are positive physical constants and 
$u$ depends on the physical context. 
${\mathcal F}$ represents a nonlinear forcing
term. 
Typically, the nonlinear problem is posed locally on a closed
bounded time interval, $J = [0, T_{1}]$, and is extended to
a maximal interval $[0, \tau)$ in ${\mathbb R}$. 
We will also consider the impact on solutions of
this system if a continuous operator constraint is adjoined. One requires,
on a subinterval of the maximal time interval $[0, \tau)$: 
\begin{equation}
\label{concons}
{\mathcal G}(u(t), u_{t}(t)) >  0.
\end{equation}
We now summarize the principal results in the article. In Proposition
\ref{contraction}
we construct the semigroup $T(t)$ with generator $U$ 
for the linear part of (\ref{cibvp}). 
Although this contractive semigroup has been extensively
investigated, 
its global stability, indeed decay, on the {\it generator} domain
appears to be recent. Generator spectral estimates are required.
In section three, we establish the convolution
formula on the frame space 
${\mathcal H}$
to define a weak solution of the 
inhomogeneous problem ${\mathcal F}(u) = {\tilde F}$ in Theorem \ref{WEUL}. 
The time interval is arbitrary, due to the global character of the
semigroup,
and facilitated by
eigenfunction approximation. Global stability on the generator domain
enters here.  
The eigenfunctions of the
Laplacian enter in another fundamental way; their invariance under the semigroup
gives rise to 
a rigorous construction of separation of variables
and a corresponding  useful approximation result, independent of the existence results. 
The relevant theorem is Theorem \ref{EIG}.
In section four, we consider
nonlinear forcing in $L^{2}$. 
We define a local solution in Theorem \ref{localsol} via
successive approximation, and use the local solution  
to define a maximal solution in Theorem \ref{maxsolex}. 
This leads to the result derived in Proposition \ref{evhmax} that every weak
solution can be extended to a maximal solution. 
In addition, we prove uniqueness in the case of local Lipschitz forcing 
in Proposition \ref{Lipunique}.
In this section, we also consider an application of the results to
the constrained case and obtain a maximal subinterval on which the
constraint holds in Theorem \ref{conexist}. The constraint may be applied to the
pair $(u, u_{t})$. 
Summary remarks and an appendix close the article.
\section{The Semigroup}
We discuss an explicit construction of the semigroup associated with the
linear part of the equation in (\ref{cibvp}). 
We will develop special features necessary
for the analysis given here. The semigroup $T(t)$ will initially be
defined on a Hilbert space ${\mathcal H}$. Much is known about the norm of 
$T(t)$ on this space. For example, 
the time decay of the semigroup for manifolds without boundary was established
in \cite{Rauch}; the results 
of these authors also hold in one dimension for intervals. 
The stability of the semigroup is under study by many mathematicians.  
We are particularly interested here in solutions contained in
the domain $D(U)$ of the semigroup generator. Instrumental to this is the
study of the norm of $T(t)$ when restricted to this smoother space. 
This distinction assumes especial importance in the work of Kato
\cite[Ch.\ 6]{J2} in his construction of the evolution operator.

After a standard definition, we
will concisely state and derive the required results. We will make
extensive use of parts (2,3) of Proposition \ref{contraction}, stated
below. 

Define $v = u_{t}$, and 
\begin{equation}
\label{defU}
U = \left[\begin{array}{cc}                                                  
0&{\mathcal I}\\
\kappa \nabla^{2} & -\nu {\mathcal I} 
 \end{array}\right]. 
\end{equation}
Here ${\mathcal I}$ represents the identity. Then the equation, 
$$
\left[\begin{array}{cc}                                                
u_{t}\\
 v_{t} 
\end{array}\right] =
U \left[\begin{array}{cc}                                                
u\\
v
\end{array}\right], 
$$
is a standard equivalent representation of 
$u_{tt} = -\nu u_{t} + \kappa \nabla^{2} u$.
In order to position this in an operator framework, we 
formulate the following definitions.
\begin{definition}
\label{functionsp}
Define the
function space,
$
{\mathcal H} = H^{1}_{0}(\Omega) \times L^{2}(\Omega),
$
and the domain $D(U)$ of $U$,
$
D(U) = 
(H^{2}(\Omega) \cap H^{1}_{0}(\Omega)) \times H^{1}_{0}(\Omega).
$
We shall employ an equivalent norm on $H^{1}_{0}(\Omega)$ by utilizing the inner
product:
\begin{equation}
(u,w)_{H^{1}_{0}} = \kappa (\nabla u, \nabla w)_{L^{2}}.
\label{equivnorm}
\end{equation}
For elements 
$
f_{j} =  \left[\begin{array}{cc}                                                
u_{j}\\
v_{j}
\end{array}\right] 
$
in $D(U), j=1,2$,
we employ the $D(U)$ inner product given by the sum,
\begin{equation}
(f_{1}, f_{2})_{D(U)} = 
(\nabla^{2} u_{1}, \nabla^{2} u_{2})_{L^{2}} +
(u_{1},u_{2})_{H^{1}_{0}} +(v_{1}, v_{2})_{H^{1}_{0}}.
\label{DUnorm}
\end{equation}
We employ the notation ${\mathcal D}(U)$ for the first components of the
elements of $D(U)$. The corresponding inner product is defined by
the truncation of (\ref{DUnorm}):
$$
(u_{1}, u_{2})_{{\mathcal D}(U)} = 
(\nabla^{2} u_{1}, \nabla^{2} u_{2})_{L^{2}} +
(u_{1},u_{2})_{H^{1}_{0}}.
$$
For clarity, throughout section two 
we will use vector notation to represent the
components in the Cartesian products which represent ${\mathcal H}$ and $D(U)$.
Throughout the remainder of the article, for $X$ a Banach space, and $J$ a
compact time interval, $C(J; X)$ is the usual Banach space with norm 
$\|u\|_{C(J; X)}  := \max_{t \in J} \|u(t)\|_{X}$.
\end{definition}

The following proposition contains the results needed for the subsequent 
sections.
\begin{proposition}
\label{contraction}
$U$ is a closed linear operator in ${\mathcal H}$. 
Denote by $\{\lambda_{n}\}$ the (positive) eigenvalues of 
$-\kappa \nabla^{2}$ on
the domain $H^{2}(\Omega) \cap H^{1}_{0}(\Omega)$.
The following properties hold.
\begin{enumerate}
\item
The resolvent set of $U$ satisfies
$\rho(U) \supset \{\lambda: 
{\rm Re}\; \lambda > \theta \}$, where, for 
$ \theta_{n} = 4 \lambda_{n} - \nu^{2}$, 
\begin{eqnarray}
\theta &=& -\nu/2, \; \mbox{if} \; \theta_{n} \geq 0, \forall n \geq 1,\\
\theta &=& \max_{n: \theta_{n} < 0} \{-\nu/2 + \sqrt{-\theta_{n}}/2 \}, \;
\mbox{otherwise}.
\label{dichotomy}
\end{eqnarray}
In particular, the imaginary axis is contained in $\rho(U)$.
\item
$U$ is the generator of a 
strongly continuous semigroup $T(t), t \geq 0$, on ${\mathcal H}$, 
which is contractive:
\begin{equation} 
\label{decayH}
\|T(t)\|_{{\mathcal H} } \leq 1, \; t \geq 0.
\end{equation}
Suppose $T_{1} > 0$ is arbitrary. 
If $G \in D(U)$, and $F \in C(J; D(U))$, where $J = [0, T_{1}]$, then
\begin{equation}
\label{represent}
V(t) = T(t)G + \int_{0}^{t} T(t-s) F(s) \; ds 
\end{equation} 
is in $C(J; D(U))$ and
satisfies the initial value problem,
\begin{equation}
\label{DWE}
V_{t} = U V(t)  + F(t), \; V(0) = G, \; 0 < t \leq T_{1}. 
\end{equation}
\item
$\|T(t)\|_{D(U)}$, defined with respect to  
$D(U)$, decays to zero as $t \rightarrow \infty$. 
In particular, there is a number $\omega$ such that this norm satisfies 
$\|T(t)\|_{D(U)} \leq \omega, t \geq 0$. 
\end{enumerate}
\end{proposition}
\begin{proof}
The property that $U$ is closed follows routinely from the definitions. 
If
$$
 \left[\begin{array}{cc}                                                
u_{n}\\
v_{n}
\end{array}\right] \rightarrow 
 \left[\begin{array}{cc}                                                
u\\
v
\end{array}\right], \;\; 
U \left[\begin{array}{cc}                                                
u_{n}\\
v_{n}
\end{array}\right] \rightarrow 
 \left[\begin{array}{cc}                                                
w\\
z
\end{array}\right], \; \mbox {in} \; {\mathcal H}, \; 
 \left[\begin{array}{cc}                                                
u_{n}\\
v_{n}
\end{array}\right] \; \; \in D(U), 
$$
then we conclude directly that $v=w$ and the $L^{2}$ limit of $\kappa
\nabla^{2} u_{n}$ is $z + \nu w$. 
From this relation, we conclude that $\nabla u$ has components with  
$L^{2}$ derivative, so that
$$
 \left[\begin{array}{cc}                                                
u\\
v
\end{array}\right] \; \in D(U), 
\; \mbox{with image} \; 
 \left[\begin{array}{cc}                                                
w\\
z
\end{array}\right].
$$
It follows that $U$ is closed. 

To prove statement (1), suppose that Re$\; \lambda > \theta$, and consider
the formal system,
$$
(\lambda {\mathcal I}  -U) \left[\begin{array}{cc}                                            
u\\
v
\end{array}\right] 
 = \left[\begin{array}{cc}                                                
w\\
z
\end{array}\right].
$$
Since $\lambda$ is permitted to be complex, for this part 
of the proof we interpret ${\mathcal H}$ and $D(U)$ as complex Hilbert
spaces. 
We show that
$\lambda {\mathcal I} - U$ 
is an algebraic isomorphism from $D(U)$ to ${\mathcal H}$, i.e.\thinspace,
the formal system is uniquely solvable.
By using the definition of $U$, we see that 
$v = \lambda u -w$, provided
$u$ is determined. 
This means that it is necessary to show that the 
following differential
equation, with homogeneous boundary values, 
determines a unique solution $u$ in $H^{2}(\Omega) \cap H_{0}^{1}(\Omega)$:
$
-\kappa \nabla^{2} u + \lambda (\lambda + \nu)u = z + (\lambda + \nu)w, 
$
for this pair $w,z$.
Now the spectrum of the self-adjoint operator $- \kappa \nabla^{2}$ consists of
discrete positive eigenvalues $\lambda_{n}$, with limit $\infty$.   
Thus, it must be shown that the product
$-\lambda (\lambda + \nu)$ excludes these numbers. A calculation, based
upon proof by contradiction, 
shows that this is achieved for 
$
\mbox{Re}\; \lambda > \theta. 
$ 
We conclude that
$\lambda {\mathcal I} - U$ is surjective and injective from $D(U)$ to
${\mathcal H}$. It is, by the definition of norms, a bounded linear
operator between these spaces. 
Its algebraic inverse, $R(\lambda, U)$, is 
a bounded linear operator, as follows from the open mapping theorem.
Indeed, a bound on the $D(U)$ norm, which dominates the ${\mathcal H}$
norm, is obtained.

We now prove part (2), and operate within real Hilbert spaces. 
We verify that the resolvent of $U$, 
$R(\lambda, U)$,
satisfies the norm inequality in ${\mathcal H}$,
\begin{equation}
\|R(\lambda, U)\| \leq \frac{1}{\lambda}, \; \lambda > 0.
\label{resolvent}
\end{equation}
The Hille-Yosida theorem \cite{BB,EN} then implies that $U$ generates a
contraction semigroup.
To establish (\ref{resolvent}), it is sufficient to show 
that the inner products 
$(Uf, f) = (f, Uf)$ are nonpositive, as is seen by expanding the square of
the ${\mathcal H}$-norms of both sides of $(\lambda -\theta) f - Uf = g$.
Now we compute, for 
 $f=\left[\begin{array}{cc}                                                
u\\
v
\end{array}\right],$
$$
(Uf, f)_{{\mathcal H}} = (v, u)_{H^{1}_{0}} + (\kappa \nabla^{2} u -
\nu v, v)_{L^{2}}.
$$
After integration by parts, the rhs reduces to $-\nu (v, v)_{L^{2}}$,
which is nonpositive. Note that the cancelation involved here depends on
the equivalent norm introduced in (\ref{equivnorm}). 

The statement that (\ref{represent}) is well defined and provides a
solution of (\ref{DWE}) depends fundamentally on the characterization of
$D(U)$ in terms of the limit of difference quotients of $T(t)$.  
The result follows by direct computation. The details of the computation
may be found in \cite[Section 6.4]{J2}, in the general case of the
evolution operator ${\mathcal U}(t,s)$. 
For the current result, ${\mathcal U}(t,s) = T(t-s)$.

To establish part (3), we use the lemma cited at the conclusion of the
proof, in conjunction with part (1). This concludes the proof. 
\end{proof}
The following lemma is quoted from
\cite[Theorem 3]{S}. 
\begin{lemma}
A bounded $C_{0}$-semigroup $e^{At}$ on a Banach space $X$ with 
(unbounded) generator
$A$ satisfies
$$
\|e^{At} (A - \lambda {\mathcal I})^{-1}\| \rightarrow 0, \; \mbox{as} \; t
\rightarrow \infty, \lambda \in \rho(A), 
$$
if and only if the imaginary axis is in the resolvent set of $A$: 
$i {\mathbb R} \subset \rho(A). 
$
\end{lemma}
An equivalent form of this result was obtained earlier in \cite{BD}.
The special case of this result for $\lambda = 0$ was obtained in
\cite{B,B2} and \cite{Phong}.  
\section{Weak Solutions for the Inhomogeneous Equation}
In this section, 
we will construct weak solutions of the special case  
of (\ref{cibvp}), usually referred to as the linear inhomogeneous
equation.
This requires the introduction of the appropriate approximation spaces. We
begin with these.
\subsection{Approximation subspaces and projections}

The eigenfunctions of the `Laplacian' $- \kappa \nabla^{2}$,
denoted $\{\phi_{n}\}_{n \geq 1}$,  will occupy a
significant role. The system is   
a complete orthogonal system in $L^{2}(\Omega)$ and $H^{1}_{0}(\Omega)$,
and (normalized to be) orthonormal in $L^{2}(\Omega)$.
A comprehensive survey of Laplacian eigenfunctions for important geometric
domains may be found in \cite{GN}.
In this section, $J = [0, T_{1}]$, where $T_{1}$ is an arbitrary terminal
time.
\begin{definition}
\label{Fourier}
Denote by $Q_{n}$ the orthogonal projection in $H^{1}_{0}(\Omega)$ onto the
linear span ${\mathcal M}_{n}$ of $\{\phi_{k}\}_{1 \leq k \leq n}$,
and by ${\tilde Q}_{n}$ the corresponding orthogonal projection in
$L^{2}(\Omega)$. Further,
denote by ${\mathcal P}_{n}$ the (closed) linear subspace of $C(J;
H^{1}_{0}(\Omega))$ defined by 
$$
{\mathcal P}_{n} = \left\{\sum_{k=1}^{n} \alpha_{k}(t) \phi_{k}({\bf x}):
\alpha_{k} \in C(J), k= 1, \dots, n \right\},
$$
and write $P_{n}$ for the projection in $C(J; H^{1}_{0}(\Omega))$ onto 
${\mathcal P}_{n}$, and ${\tilde {\mathcal P}}_{n}, 
{\tilde P}_{n}$, for the spaces and projections in  
$C(J; L^{2}(\Omega))$.
\end{definition}
A direct implementation of the definition shows that $Q_{n}u$ has the same
formal representation as ${\tilde Q}_{n}u$. Furthermore, as
verified in \cite{J3}, ${\mathcal P}_{n}$ is closed, $P_{n} = P_{n}^{2}$
is a standard projection, and $P_{n}$ can be interpreted, for each fixed
$t \in J$, as the orthogonal projection $Q_{n}$ onto ${\mathcal M}_{n}$.
Similar remarks apply to ${\tilde P}_{n}$. 
Moreover, 
$P_{n}f$ converges to $f$ in  $C(J; H^{1}_{0}(\Omega))$ and  
${\tilde P}_{n}g$ converges to $g$ in
$C(J; L^{2}(\Omega))$.
These statements follow from an estimation of the maxima on $J$ of the respective error functions.

\subsection{Projections and classical solutions}  
Designate by $F_{n}$ the pair $(0, {\tilde P}_{n}{\tilde F})$ 
which is a member of $C(J; D(U))$. 
Set $G=(u_{0}, v_{0}) \in {\mathcal H}$ 
and designate by $G_{n}$ the pair
$(Q_{n} u_{0}, {\tilde Q}_{n} v_{0})$. 
It follows from Proposition \ref{contraction}, part (2), that 
\begin{equation}
\label{represent2}
V_{n}(t) = T(t) G_{n} + \int_{0}^{t} T(t-s) F_{n}(s) \;ds
\end{equation}
is a strong solution of the initial/boundary-value problem,
\begin{eqnarray}
u_{tt} &=& - \nu u_{t} + \kappa \nabla^{2} u + {\tilde P}_{n}{\tilde F}, 
\nonumber \\
u({\bf x},t)  &=& 0, \; {\bf x} \in \partial \Omega, 0 < t \leq T_{1}, 
\nonumber \\ 
u({\bf x}, 0) = Q_{n} u_{0}({\bf x}), && 
u_{t}({\bf x},0) = {\tilde Q}_{n} v_{0}({\bf x}), \; {\bf x} \in \Omega. 
\label{clapp}
\end{eqnarray}
The $V_{n}$ will serve as approximations to establish existence of weak
solutions on $J$ in section \ref{sec3.2}, and will be shown to have
eigenfunction components in section \ref{sec3.3}.  
\subsection{Weak solutions for the linear inhomogeneous equation}
\label{sec3.2}

\begin{definition}
\label{weaksol}
For $T_{1} > 0$ arbitrary, 
define $J = [0, T_{1}]$ and suppose ${\tilde F} \in C(J; L^{2}(\Omega))$ and 
$(u_{0}, v_{0}) \in {\mathcal H}$ are given.
Suppose there exists 
$u$ such that: 
\begin{enumerate}
\item
$(u, u_{t})$ is 
continuous from $J$
to ${\mathcal H}$, with $u_{tt}$ continuous from $J$ to
$H^{-1}(\Omega)$.
\item
The initial conditions are satisfied:
$(u(0, {\bf x}), u_{t}(0,{\bf x})) = (u_{0}({\bf x}), v_{0}({\bf x})), 
\; {\bf x} \in \Omega$. 
\item
$\forall \; \phi \in C(J; H^{1}_{0}(\Omega)), 
\forall \; 0 < t \leq T_{1}$, 
\begin{equation}
(u_{tt}, \phi)_{L^{2}}  = 
-\nu(u_{t}, \phi)_{L^{2}}  - \kappa (\nabla u,\nabla \phi)_{L^{2}}  + 
({\tilde F},\phi)_{L^{2}}.  
\label{wcibvp}
\end{equation}
\end{enumerate}
Then we say that $u$ is a weak solution for the linear 
inhomogeneous equation, i.\thinspace e.\thinspace, the version of
the system
(\ref{cibvp}) with ${\mathcal F}(u) = {\tilde F}$.
\end{definition}
\begin{theorem}
\label{WEUL}
Given ${\tilde F} \in C(J; L^{2})$, and $G = (u_{0}, v_{0}) \in {\mathcal H}$, 
the pair given by 
\begin{equation}
\label{Katorepw}
V(t) = T(t)G + 
\int_{0}^{t} T(t-s) (0, {\tilde F}(s)) \;ds
\end{equation}
is a unique weak solution, as defined in Definition
\ref{weaksol}.
\end{theorem}
\begin{proof}
The classical solutions $V_{n}$ as defined by (\ref{represent2}) are seen to be 
weak solutions of the system defined by projection, 
upon application of the Gauss-Green theorem, 
and converge uniformly
in ${\mathcal H}$ to $V$. This convergence makes direct use of the
contractive semigroup properties of $T(t)$ and the projection properties. 
To verify this, suppose that 
$(u_{n}, (u_{n})_{t}) = V_{n} \in D(U)$ is given by 
(\ref{represent2}). For $\phi \in C(J; C^{\infty}_{0}({\mathbb R}^{N}))$, 
we obtain upon direct multiplication and integration,
\begin{equation}
((u_{n})_{tt}, \phi)_{L^{2}}  = 
-\nu((u_{n})_{t}, \phi)_{L^{2}}  
+ \kappa (\Delta u_{n}, \phi)_{L^{2}}  + 
(F_{n},\phi)_{L^{2}}.  
\label{wcibvp5}
\end{equation}
An application of the Gauss-Green theorem, followed by the limit $n
\rightarrow \infty$, gives
\begin{equation}
(W, \phi)_{L^{2}}  = 
-\nu(u_{t}, \phi)_{L^{2}}  - \kappa (\nabla u,\nabla \phi)_{L^{2}}  + 
({\tilde F},\phi)_{L^{2}}.  
\label{wcibvp6}
\end{equation}
Here we have used the convergence of $V_{n}$ to $V$ in $C(J; {\mathcal
H})$. Also, we have identified $W$ with a member of $C(J; H^{-1})$.
Further,
$$
(u_{n})_{tt} \rightarrow W, \; \mbox{in} \; C(J; H^{-1}). 
$$
Upon integration, it follows that 
$$
(u_{n})_{t} -{\tilde Q}_{n} v_{0}  \rightarrow \int_{0}^{t} W \;ds,
\; \mbox{in} \; C(J; H^{-1}). 
$$
Independently, we have established that $(u_{n})_{t} \rightarrow u_{t}$ 
in $C(J; L^{2})$ 
and 
${\tilde Q}_{n} v_{0} \rightarrow v_{0}$ in $L^{2}$
It follows that $W = u_{tt}$. This concludes the proof of existence of weak solutions upon use of 
the denseness of the test functions
$C^{\infty}_{0}(\Omega)$.

In order to verify uniqueness, we make some preliminary statements.
Notice that the difference of weak solutions of the inhomogeneous initial/boundary-value problem 
is a weak solution of the homogeneous problem with homogeneous initial conditions. This has only the zero solution. Indeed,
suppose $f \in C(J; {\mathcal H})$ 
is a weak solution of $df/dt = Uf$, with homogeneous initial conditions. 
For the time integrated version, we may take $L^{2}$ inner products:
$$
(f(t), f(t))_{L^{2}} = \int_{0}^{t} (Uf(s), f(s))_{L^{2}} \; ds.
$$ 
The rhs is nonpositive, as shown in section two in the proof of part (2) of Proposition \ref{contraction}. 
Since the lhs is nonnegative, we conclude that $f = 0$.
This establishes uniqueness of weak solutions.
\end{proof}
\subsection{Approximation by eigenfunctions}
\label{sec3.3}
In this section, we examine more closely the explicit role of
eigenfunction approximation. 
The section is independent of the rest of the article, and is included to show that
when approximation of the inhomogeneous term is carried out by members of ${\tilde {\mathcal P}}_{n}$, the
solution formula generates approximations from this space. 
This is accomplished by a rigorous
implementation of classical separation of variables. In particular, by
direct calculation, we verify the semigroup invariance on the
approximation spaces.
\begin{definition}
\label{seppde}
We say that $\psi_{n}(t) \phi_{n}({\bf x}) $ is a separated solution of the damped wave equation, 
\begin{equation*}
u_{tt} = - \nu u_{t} + \kappa \nabla^{2} u, 
\end{equation*}
with homogeneous
boundary conditions and zero forcing if $\psi_{n}$ satisfies the ordinary differential equation,
\begin{equation}
\label{sepode}
\psi_{n}^{\prime \prime} + \nu \psi_{n}^{\prime} + \lambda_{n} \psi_{n} = 0, 
\end{equation}
on the time interval $t > 0$. Here, $(\lambda_{n}, \phi_{n})$ are an eigenvalue/eigenfunction pair for the operator
$-\kappa \nabla^{2}$ with Dirichlet boundary conditions on $\Omega$.
\end{definition}
\begin{lemma}
Separated solutions satisfy the damped wave equation as strong solutions and are members of the regularity class
$C([0, \infty); D(U))$.  
They can be explicitly written as follows.
\begin{enumerate}
\item
$4 \lambda_{k} - \nu^{2} > 0$.
In this case, the separated solutions are of the form,
$(a C_{k}(t) + b S_{k}(t)) \phi_{k}(x)$, where
\begin{eqnarray*}
C_{k}(t) :=\exp(-(\nu/2)t) \cos (\omega_{k}t), &&  
S_{k}(t) :=\exp(-(\nu/2)t) \sin (\omega_{k}t), \\
dC_{k}/dt = (-\nu/2) C_{k}(t) - \omega_{k} S_{k}(t), && dS_{k}/dt = (-\nu/2) S_{k}(t) + \omega_{k} C_{k}(t), 
\end{eqnarray*}
$$
\omega_{k} =\frac{\sqrt{4 \lambda_{k} 
 -\nu^{2}}}{2}. 
$$  
\item
$4 \lambda_{k} - \nu^{2} = 0$.
In this case, the separated solutions are constructed by replacing $C_{k}$
and $S_{k}$ by $c_{k}$ and $s_{k}$, where 
\begin{eqnarray*}
c_{k}(t) :=\exp(-(\nu/2)t), && 
s_{k}(t) := t\exp(-(\nu/2)t), \\
dc_{k}/dt = (-\nu/2) c_{k}(t), && ds_{k}/dt = c_{k}(t) - (\nu/2) s_{k}(t).
\end{eqnarray*}
\item
$4 \lambda_{k} - \nu^{2} < 0$.
Here, the replacements are $D_{k}$ and $E_{k}$, where
\begin{eqnarray*}
D_{k}(t) :=\exp((-\nu/2 + \rho_{k})t), && 
E_{k}(t) := \exp((-\nu/2 -\rho_{k})t), \\
dD_{k}/dt = (-\nu/2 + \rho_{k}) D_{k}(t), && dE_{k}/dt = 
(-\nu/2 - \rho_{k}) E_{k}(t).
\end{eqnarray*}
Here, $\rho_{k} = 
\sqrt{\nu^{2} - 4 \lambda_{k}}/2$. 
\end{enumerate}
\end{lemma}
\begin{proof}
The regularity class for the separated solutions follows directly from their definition, given in the statement of
the lemma. 
A direct substitution into the damped wave equation of $\psi_{n} \phi_{n}$ verifies the first statement 
when (\ref{sepode}) is employed for $\psi_{n}$ and the eigenfunction property is employed for $\phi_{n}$. The
stated solution trichotomy depends upon the solutions of the associated characteristic equation,
$$
m^{2} + \nu m + \lambda_{n} = 0,
$$ 
which leads to the above cases, depending on the discriminant $D = \nu^{2} - 4 \lambda_{n}$. These are determined
by $D < 0, D = 0, D > 0$, resp. 
\end{proof}
\begin{lemma}
Suppose that the pair $(0, \phi_{n})$ is given as an initial condition for the damped wave equation. There is a
unique separated solution pair $(u_{\ast}, v_{\ast})$, with $v_{\ast} = du_{\ast}/dt$,
such that the pair $(u_{\ast}, v_{\ast})$ assumes the evaluation  
$(0, \phi_{n})$ at $t = 0$. In particular, this pair is precisely $T(t) (0, \phi_{n})$.  
We have the following formulas. 
\begin{enumerate}
\item
$4 \lambda_{k} - \nu^{2} > 0$.
We have
$T(t) (0, \phi_{k}) = (u_{\ast}, v_{\ast})$, where 
$$
u_{\ast}({\bf x},t) = (1/\omega_{k}) S_{k}(t) \phi_{k}({\bf x}), \;
v_{\ast}({\bf x},t) = (C_{k}(t) - \nu/(2 \omega_{k}) S_{k}(t))
\phi_{k}({\bf x}).  
$$
\item
$4 \lambda_{k} - \nu^{2} = 0$.
We have
$T(t) (0, \phi_{k}) = (u_{\ast}, v_{\ast})$, where 
$$
u_{\ast}({\bf x},t) =  s_{k}(t) \phi_{k}({\bf x}), \;
v_{\ast}({\bf x},t) = (c_{k}(t) - (\nu/2) s_{k}(t))
\phi_{k}({\bf x}).  
$$
\item
$4 \lambda_{k} - \nu^{2} < 0$.
We have
$T(t) (0, \phi_{k}) = (u_{\ast}, v_{\ast})$, where 
$$
u_{\ast}({\bf x},t) =  1/(2 \rho_{k})(D_{k}(t) -  E_{k}(t)) \phi_{k}({\bf x}), \;
$$
$$
v_{\ast}({\bf x},t) = (1/2)[(1 - \nu/(2 \rho_{k})) D_{k}(t) 
+ (1 +  \nu/(2 \rho_{k})) E_{k}(t)]
\phi_{k}({\bf x}).  
$$
\end{enumerate}
\end{lemma}
\begin{proof}
The semigroup is uniquely determined by the conditions that 
$T(0) = {\mathcal I}, \;(d/dt) T(t) = UT(t)$. 
The formulas cited, derived from the previous lemma, satisfy these conditions when evaluated at $(0, \phi_{n})$.
\end{proof}
\begin{theorem}
\label{EIG}
The components of $V_{n}$ are in $C(J; {\mathcal P}_{n} \times
{\tilde {\mathcal P}}_{n}) $.
In particular, any weak solution $V$ of the linear inhomogeneous equation  
can be approximated in $C(J; {\mathcal H})$ by eigenfunctions via a
standard Galerkin procedure.
\end{theorem}
\begin{proof}
The previous lemmas have allowed us 
to examine the action of $T(t)$ on $(0, \phi_{k})$. 
Linearity of the semigroup $T$ permits extension to members of the form 
$\sum_{k=1}^{n}\alpha_{k}(s) (0,\phi_{k}(x))$. Indeed, motivated by the representation for the semigroup action, we
compute
$$
T(t-s)
\sum_{k=1}^{n}\alpha_{k}(s) (0,\phi_{k}({\bf x})) = 
\sum_{k=1}^{n}\alpha_{k}(s) T(t-s) (0,\phi_{k}({\bf x}))  
$$
$$
= \sum_{k=1}^{n}\alpha_{k}(s) (\beta_{k}(t-s), \gamma_{k}(t-s))\phi_{k}({\bf x}),  
$$  
where the functions $\beta_{k}, \gamma_{k}$ are specified in the preceding lemma.
In all three cases, 
since formula (\ref{represent2}) involves integration in the variable $s$, 
we conclude that both components are functions from $J$ to
${\mathcal M}_{n}$.
This argument also holds for the approximation of the initial conditions.
It follows that the convergence of $V_{n}$ to $V$ can be interpreted as
convergence of the Faedo-Galerkin method.
\end{proof}
\begin{remark}
As mentioned earlier, the other results of this article are independent of the preceding theorem.  
\end{remark}
\subsection{Solution stability as $t \rightarrow \infty$}
Suppose ${\tilde F} \in C([0, \infty); L^{2})$. 
We give a sufficient condition on the norm growth of ${\tilde F}$ so that
the ${\mathcal H}$ norm of $V$,
defined on each interval $[0, T_{1}]$ by Theorem \ref{WEUL}, 
remains bounded in norm as $t \rightarrow \infty$.
\begin{corollary}
Suppose ${\tilde F} \in 
C([0, \infty); L^{2}) \cap 
L^{1}([0, \infty); L^{2})$.
Then $\|V\|_{C([0, T_{1}]; {\mathcal H})}$
remains bounded as  $T_{1}
\rightarrow  \infty$.  
\end{corollary}
\begin{proof}
This is implied by the contractive property of the semigroup on ${\mathcal
H}$.
\end{proof}
\begin{remark}
A similar result can be formulated for strong solutions, with the
appropriate assumptions on $u_{0}, v_{0}, F$. Interestingly, because of
the decay of $\|T(t)\|_{D(U)}$, one has the additional asymptotic property
that the term in $V(t)$ due to the initial data decays to zero as 
$t \rightarrow \infty$. 
\end{remark}
\section{The Nonlinear System with Operator Forcing}
In this section we construct weak solutions for  
the general nonlinear system (\ref{cibvp}). We begin by defining the
properties of the forcing term.
\begin{definition}
\label{defFforce}
We consider an operator ${\mathcal F}$ with the following properties.
\begin{enumerate}
\item
${\mathcal F}$ is defined on $L^{2}(\Omega)$ with range in $L^{2}(\Omega)$.
\item
${\mathcal F}$ is locally bounded: 
For every bounded set $B$ in $L^{2}(\Omega)$, 
${\mathcal F}(B)$ is bounded in
$L^{2}(\Omega)$. 
\item
${\mathcal F}$ is continuous from $L^{2}(\Omega)$ to $L^{2}(\Omega)$. 
\end{enumerate}
\end{definition}
\begin{remark}
We observe that Lipschitz continuity, even local Lipschitz continuity, 
is not assumed for ${\mathcal F}$, since many 
significant applications do not satisfy this property.  
The preceding hypothesis is discussed by Krasnosel'skii \cite[Th.\ 2.1,
p.\ 22]{Kras}. The class of Nemytskii operators with growth conditions is known to imply these
conditions \cite{Show}. 
\end{remark}
\begin{remark}
\label{extend}
The hypotheses for the nonlinear forcing term are not as restrictive as first appears. Although forcing by
function composition is typical in the literature, other types of forcing are permitted, including convolution
such as occurs in mollification. For example, ${\mathcal F}(u) = \phi_{\epsilon} \ast |u|^{2}$ satisfies the
required hypotheses, where $\phi_{\epsilon}$ is a classical smooth mollifier. In fact, by Young's convolution
inequality, one sees that ${\mathcal F}$ is locally Lipschitz continuous on $L^{2}$. When appropriate compactness
is available, one can use this theory to extend results to nonlinearities stronger than those which can be treated
by Nemytskii operators.    
\end{remark}
\begin{definition}
\label{NWS}
Suppose there exist $T_{1}$ and $u$ such that
\begin{enumerate}
\item
$(u, u_{t})$ is 
continuous from $J$
to ${\mathcal H}$, with $u_{tt}$ continuous from $J$ to
$H^{-1}(\Omega)$.
\item
The initial conditions are satisfied:
$(u({\bf x}, 0), u_{t}({\bf x}, 0))  = 
(u_{0}({\bf x}), v_{0}({\bf x})), \; 
\mbox{\rm for} \; {\bf x} \in \Omega$. \nonumber \\
\item
$\forall \; \phi \in C(J; H^{1}_{0}(\Omega)), 
\forall \; 0 < t \leq T_{1}$, 
\begin{eqnarray}
(u_{tt}, \phi)_{L^{2}}  &=& 
-\nu(u_{t}, \phi)_{L^{2}}  - \kappa (\nabla u,\nabla \phi)_{L^{2}}  + 
({\mathcal F}(u),\phi)_{L^{2}}.  
\label{cibvp2}
\end{eqnarray}
\end{enumerate}
Then $u$ is said to be a weak solution of the 
system (\ref{cibvp}).
If $u$ is defined on an interval $[0, \tau)$, with the property that its
restriction to any closed interval $[0, T_{1}]$ is a weak solution, then 
$u$ is called maximal if it has no proper extension which is a weak solution. 
\end{definition}

\subsection{Operator properties}
We begin with a definition, which identifies two fundamental constants.
\begin{definition}
Define $\omega_{0}$:
\begin{equation}
\label{Floc}
\omega_{0} = \sup_{u: \|u - u_{0}\|_{L^{2}} \leq 1} 
\|{\mathcal F}(u)\|_{L^{2}}.
\end{equation}
Define $s_{0}$ (by continuity): 
\begin{equation}
\label{szero}
\|T(t)(u_{0}, v_{0}) - (u_{0}, v_{0}) \|_{{\mathcal H}} \leq 1/2, 
\; 0 \leq t \leq s_{0}.
\end{equation}
\end{definition}
\begin{proposition}
\label{invV}
There is a time $t_{1} > 0$ such that 
the operator,
$$
{\mathcal V}(u,v)(t) = V(t), 
0 \leq t \leq t_{1}, 
$$
where $V(t)$ is defined by Theorem \ref{WEUL} with $T_{1} = t_{1}$,
and ${\tilde F}(t) = {\mathcal F}(u(t))$,
maps 
$$
{\mathcal B}_{0} = \{(u,v) \in C(J_{1}; {\mathcal H}): 
\|(u,v) - (u_{0}, v_{0})\|_{C(J_{1}; {\mathcal H})}
\leq 1\}  
$$
into itself. Here, $J_{1} = [0, t_{1}]$. 
Furthermore, 
the
operator ${\mathcal V}: {\mathcal B}_{0} \mapsto {\mathcal B}_{0}$ 
is uniformly equicontinuous on $J_{1}$ 
in the sense that its range is a
uniformly equicontinuous family on $J_{1}$.
The time $t_{1}$ can be explicitly represented as,
\begin{equation}
\label{deft1}
t_{1} = \min (s_{0}, 1/(2 \omega_{0})).
\end{equation}
\end{proposition}
\begin{proof}
Since $T_{1}$ as used in the previous section is arbitrary, it follows
that the mapping ${\mathcal V}$ is well-defined, for any choice of $t_{1}$.
Theorem \ref{WEUL} is also used to verify the invariance statement. 
Indeed, we estimate $\|V(t)- G\|_{{\mathcal H}}$ as follows. 
For each $ t \in J_{1}$,
\begin{equation*}
\|V(t)- G\|_{{\mathcal H}} \leq 
\|T(t)G - G\|_{{\mathcal H}} + \int_{0}^{t}  
\|{\mathcal F}(u(s))\|_{L^{2}} \; ds.
\end{equation*}
Here, we have used the contractive property of $T$ on ${\mathcal H}$.
The first term is bounded by $1/2$ by (\ref{szero}), while the second term
is also bounded by $1/2$ by the combination of (\ref{Floc}) and
(\ref{deft1}). In order to establish uniform equicontinuity, let
$0 \leq \tau_{1} < \tau_{2} \leq t_{1} $ be given in $J_{1}$ and
$(u, v) \in {\mathcal B}_{0}$. We first write the representation,
$$
V(\tau_{2}) - V(\tau_{1}) = [T(\tau_{2}) - T(\tau_{1})]G \;
+ \int_{0}^{\tau_{1}}[T(\tau_{2} - s) - T(\tau_{1} - s)] 
(0, {\tilde F}(s)) \; ds
$$
$$
+ \int_{\tau_{1}}^{\tau_{2}}T(\tau_{2} - s) (0, {\tilde F}(s)) \; ds. 
$$ 
As previously, we use (\ref{Floc}) 
and the contractive property of $T$ to
conclude that 
$$
\|V(\tau_{2}) - V(\tau_{1}) \|_{{\mathcal H}} \leq \|T(\tau_{2}) -
T(\tau_{1})\|_{{\mathcal H}} \|G\|_{{\mathcal H}} 
$$
$$
 +  (\omega_{0}/2)[t_{1}  \sup_{0 \leq t \leq t_{1} - \delta} 
\|T(t+\delta) - T(t)\|_{{\mathcal H}} + \delta]. 
$$
Here, $\delta = \tau_{2} - \tau_{1}$. When the uniform operator continuity property of $\|T(t)\|_{{\mathcal H}}$ is utilized, one obtains
the uniform equicontinuity.
\end{proof}
\subsection{The local solution and the maximal solution}
We will use the method of successive approximation to establish the
existence of a local solution. Although it would be possible to use the
Schauder fixed point theorem on the eigenspaces, followed by a convergence
analysis, the use of successive approximation appears to be a more
constructive approach.
Moreover, it provides a proof independent of Theorem \ref{EIG}.
\begin{theorem}
\label{localsol}
There is a weak solution 
satisfying (\ref{cibvp2}) on $[0, t_{1}]$.
Here, $t_{1}$ is defined in (\ref{deft1}). The solution may be obtained by
subsequential convergence, derived from successive solution of linear
problems (see (\ref{successive}) below).
\end{theorem}
\begin{proof}
Define the sequence of successive approximations, 
\begin{equation}
(u_{k}, v_{k}) = {\mathcal V}(u_{k-1},
v_{k-1}), k \geq 1.  
\label{successive}
\end{equation}
The mapping ${\mathcal V}$ is defined in Proposition \ref{invV}, and its
fixed points are weak solutions.
This sequence is contained in ${\mathcal B}_{0}$, hence is bounded in ${\mathcal H}$, uniformly in
$t \in J_{1}$.
By the first proposition  of the
appendix, we 
obtain a subsequence, weakly convergent in ${\mathcal H}$ pointwise in $t \in J_{1}$, to an element,
$(u, v)$. This follows from the identifications $X = Y = {\mathcal H}$ and the equicontinuity established in
Proposition \ref{invV}.
We claim that  
$(u, v)$ is a fixed point, ${\mathcal V}(u,v) = (u, v)$, hence a weak
solution.
The principal question is the uniform $L^{2}$ convergence in $t$ of $u_{k}$ 
to $u$. 
This sequence is already known to be 
$L^{2}$ convergent, pointwise  in $t$,
by the Rellich theorem.
We will show that a subsequence of $u_{k}$ is convergent to $u$ in
$C(J_{1}; L^{2})$, which allows the continuity of ${\mathcal F}$ to be
applied 
within the rhs of the representation sequence for ${\mathcal V}$.
In fact,
the solutions of the inhomogeneous equations are given by (cf.\ Theorem \ref{WEUL})
\begin{equation}
{\mathcal V}(u_{k}, v_{k})(t) = T(t)G + 
\int_{0}^{t} T(t-s) (0, {\tilde F}(u_{k-1})(s)) \;ds
\end{equation}
The rhs converges in $C(J_{1}; L^{2})$ to ${\mathcal V}(u,v)$, 
under the assumption of the uniform convergence of 
$u_{k-1}(s)$. We establish this now.
Consider the decomposition,
$$
u - u_{k} = (u - {\tilde P}_{n}u) + ({\tilde P}_{n}u - {\tilde P}_{n}u_{k})
+ ({\tilde P}_{n}u_{k} - u_{k}).
$$
Suppose a sequence $\epsilon_{\ell} \rightarrow 0$ is specified.
The first term can be estimated in $C(J_{1}; L^{2})$ for sufficiently large
$n$, not to exceed $\epsilon_{\ell}/3$. It can be shown (see the following
Lemma \ref{nwidth}) that a similar statement holds, 
uniformly in $k$, for the third term. 
The estimation of the second term is more complicated, and is carried out
by a variant of Cantor's diagonalization argument. We first set up the
tableau $\{{\mathcal R}_{j}\}_{j \geq 1}$ of nested sequences 
${\mathcal R}_{j}$. These will be chosen to have the property that  
$$
\lim_{m \rightarrow \infty} {\tilde P}_{j} u_{jm}={\tilde P}_{j}u, \; j \geq 1. 
$$
Here, $\{{\mathcal R}_{j}\} = \{u_{jm}\}_{m \geq 1}, \; j =1,2, \dots$.
Suppose that the first $n-1$ nested rows ${\mathcal R}_{j}$ of the
tableau have been defined, with the stated property. 
Observe that 
${\tilde P}_{n}{\mathcal R}_{n-1}$
is a precompact set in $C(J_{1}; L^{2})$, 
which is a consequence of 
Lemma \ref{precompact} to follow.
Select a convergent subsequence 
${\tilde P}_{n} {\mathcal R}_{n}$. 
Pointwise convergence established earlier shows that the 
necessary limit is ${\tilde P}_{n}u$. 
Now define the $n$th row of the tableau to be  
$\{{\mathcal R}_{n}\}$. 
Proceed inductively to obtain the tableau. 

We are now ready for the variant of Cantor selection. 
Define (the subsequence) $\{u_{n_{\ell}} \}$ as follows. Given
$\epsilon_{\ell}$, choose $n = n_{\ell}$ such that the first and third
terms of the decomposition are estimated, as discussed earlier. 
The appropriate $\ell$th element of the sequence is selected from row
$n_{\ell}$ so that its projection (and those of its successors) lies within
$\epsilon_{\ell}/3$ in norm, of the projection of $u$. 
This inductively defines a subsequence of 
(the first components of) the original
weakly convergent sequence, which is uniformly convergent when viewed in
$L^{2}$. Uniform convergence is preserved when ${\mathcal F}$ is applied.
Earlier arguments establish the fixed point.
\end{proof}
We now state and prove the lemmas used in the preceding proof.
\begin{lemma}
\label{precompact}
Let $(u_{k}, v_{k})$ be a sequence in ${\mathcal V}({\mathcal B}_{0})$. Then, for each fixed positive integer $n$,
$$
\{{\tilde P}_{n} u_{k}\}
$$
is a precompact set in $C(J_{1}; L^{2})$.
\end{lemma}
\begin{proof}
We shall make use of the second proposition of the appendix. 
Define $B = \{u \in H^{1}_{0}: \|u - u_{0}\|_{H^{1}_{0}} \leq 1 \}$. 
Now make the identifications, 
$$X = J_{1}, \; Y = L^{2}, \;  
C = {\tilde Q}_{n}B, \; {\mathcal F}_{0} =  
\{{\tilde P}_{n} u_{k}\}.
$$
The uniform equicontinuity has been established in 
Proposition \ref{invV} in the (stronger) case where the ${\mathcal H}$ norm is employed.  
It therefore holds for $L^{2}$. 
Since $C$ is a closed bounded subset of a finite dimensional space, it is compact as required.The
precompactness now follows from Proposition A.2.  
\end{proof}
\begin{lemma}
\label{nwidth}
Suppose that $\{h_{k}\}$ is a sequence which is bounded in $C(J;
H^{1}_{0}(\Omega))$. 
Then, uniformly in $k$,  
${\tilde P}_{n} h_{k} \rightarrow h_{k}, \; n \rightarrow \infty, \; \mbox{in} \; C(J_{1}; L^{2})$.
\end{lemma}
\begin{proof}
Suppose that $b$ is a bound for $h_{k}$ in $C(J, H^{1}_{0}(\Omega))$. 
For each fixed $t$, consider the
self-adjoint operator $R = -(\kappa/b^{2})\nabla^{2}$ on ${\mathcal D}(U)$, 
and the
ellipsoid in $L^{2}(\Omega)$ defined by
$$
{\mathcal R}(t) = \{u \in {\mathcal D}(U): (Ru, u)_{L^{2}} \leq 1\}.
$$ 
Then the $L^{2}(\Omega)\; n$-width of ${\mathcal R}(t)$ 
is attained by the subspace 
${\mathcal M}_{n}$ for each fixed $t$ \cite{J1}, so that, by closure,  
this holds true for the widths $d_{n}$ 
of the closed $H^{1}_{0}(\Omega)$-ball of radius
$b$. The latter are directly computable \cite{J1} in terms of the eigenvalues of
$R$: $d_{n} = b \lambda_{n+1}^{-1/2}$, where we retain the earlier meaning
of $\lambda_{n}$. 
This implies that, in $L^{2}(\Omega)$, uniformly in $t$, and in $k$,
${\tilde Q}_{n} h_{k} - h_{k} \rightarrow 0, \; n \rightarrow \infty$.
Equivalently, uniformly in $k$,
$$
{\tilde P}_{n} h_{k} - h_{k} \rightarrow 0, \; n \rightarrow \infty.
$$ 
\end{proof}
The following corollary follows from the preceding arguments.
\begin{corollary}
\label{extendlocal}
Suppose that $(u_{1}, v_{1}) = (u(t_{1}), v(t_{1}))$ and define 
$\omega_{1}, t_{2}, {\mathcal V}$ in analogy with the procedure of  
Theorem \ref{localsol}. This extends the solution to $[0, t_{1} + t_{2}]$, and
the procedure can be repeated by induction to any interval $[0, \tau_{k}]$,
where 
$$
\tau_{k} = \sum_{j=1}^{k} t_{j}, \; k \geq 1.
$$
\end{corollary}
\begin{proof}
Given the local solution as defined in Theorem \ref{localsol}, one again
defines a local solution on 
$[t_{1}, t_{1} + t_{2}]$ by the same method. It remains to verify that the two
local solutions are restrictions of a global solution on  
$[0, t_{1} + t_{2}]$. This is immediate from Definition \ref{NWS},
however, together with the equivalence of a weak solution with the
representation of the operator ${\mathcal V}$.
The induction proceeds similarly.
\end{proof}
\begin{theorem}
\label{maxsolex}
There is a maximal solution, i.\thinspace e.\thinspace, a solution with no
proper extension. The time interval for the maximal solution is
$[0, \tau)$, where
$$
 \tau = 
\lim_{k \rightarrow \infty} \tau_{k}.
$$
The solution, designated $V$,  
satisfies Definition \ref{NWS} on every compact subinterval
of 
$[0, \tau)$.
\end{theorem}
\begin{proof}
By Corollary \ref{extendlocal}, applied inductively for $k = 1, 2, \dots,$, we conclude that
there is a solution $V$ defined on $[0, \tau)$.
Since $[0, \tau) = \cup_{k=1}^{\infty} [0, \tau_{k}]$, 
the solution
satisfies Definition \ref{NWS} on every compact subinterval
of 
$[0, \tau)$.
The only statement which is not immediate from the construction is that
there is no proper extension of $V$, which is a weak solution. 
Suppose, for a contradiction, that $V_{\rm ext}$ is a proper extension of
$V$. 
In particular, 
the following are uniformly continuous functions on $[0, \tau]$ into $L^{2}$:
$$
u_{\mbox{\rm ext}}, {\mathcal F}(u_{\rm ext}). 
$$
Further,
\begin{equation}
\label{limomega}
V_{\rm ext}(\tau) = \lim_{k \rightarrow \infty} V(\tau_{k}).
\end{equation}
Since $\tau < \infty$, 
$$t_{k+1} - t_{k} \rightarrow 0, \; k \rightarrow \infty.$$
There are two possibilities in this case.
Either the numbers $\omega_{k}$, 
defined in (\ref{Floc}) for $k = 0$, and
selected inductively for each
interval $[t_{k}, t_{k+1}]$, must be an unbounded sequence; or, the
numbers $s_{k}$, defined in (\ref{szero}) for $s_{0}$, 
and defined more generally by the formula,
\begin{equation}
\label{sk}
\|T(t)(u_{k}, v_{k}) - (u_{k}, v_{k}) \|_{{\mathcal H}} \leq 1/2, 
\; 0 \leq t \leq s_{k},
\end{equation}
must 
possess a subsequence which converges to
zero. Neither of these possibilities can occur. The first is excluded
since 
${\mathcal F}$ 
has bounded range when evaluated 
on the compact set $K = \{V_{\rm ext}(t): 0 \leq t \leq \tau\}$.
The second is excluded since 
$\|T(t) - I\|$ is uniformly
continuous on the compact set $K$ cited above, so that $s_{k}$ may be chosen independently of $k$. 
\end{proof}
A question of interest, since uniqueness has not been resolved in
the general case of $L^{2}$ forcing, is whether every solution in the
sense of Definition \ref{NWS}, has a maximal extension. The answer is
affirmative.
\begin{proposition}
\label{evhmax}
Every solution satisfying Definition \ref{NWS} has a maximal extension. It
may be constructed by the inductive procedure defined in
Corollary \ref{extendlocal} 
and Theorem \ref{maxsolex}. 
\end{proposition}
\begin{proof}
The proof begins with the assumed solution on $[0, T_{1}]$, and proceeds
identically as in the corollary and the theorem.
\end{proof}
\subsection{Further properties I: Uniqueness}
The semigroup representation of the weak solution is sufficient to prove
uniqueness of the maximal solution if ${\mathcal F}$ is locally Lipschitz.
We have the following.
\begin{proposition}
\label{Lipunique}
If ${\mathcal F}$ is locally Lipschitz, then the maximal solution is
unique.
\end{proposition}
\begin{proof}
Suppose that there exist two distinct maximal solutions $V_{1}$ and
$V_{2}$, defined on $[0, \tau_{1}), [0, \tau_{2})$, resp. 
Consider any compact interval $J = [0, T_{1}]$ common to these intervals. 
It is immediate that the ${\mathcal H}$ norm of $V_{1} - V_{2}$ 
on $J$ 
satisfies Gronwall's inequality under the locally
Lipschitz assumption. It follows that one of these functions is a proper
extension of the other, which contradicts maximality.
\end{proof}
\subsection{Further properties II: The constrained equation}
We now add a constraint to the system
(\ref{cibvp2}). The framework permits the constraint to be applied to
$(u,u_{t})$, not simply to $u$. 
We have the following result. 
\begin{theorem}
\label{conexist}
Suppose that ${\mathcal G}$ is a continuous real-valued functional defined on 
${\mathcal H}$, which is positive for the initial values.
For any maximal solution,
there is a positive number $\tau^{\prime} \leq \tau$ such that 
$[0, \tau^{\prime})$ is maximal for the constraint. In particular, 
if $\tau^{\prime} <
\tau$, the constraint fails at $t = \tau^{\prime}$.
\end{theorem}
\begin{proof}
Define $\tau^{\prime}$ to be the 
(possibly infinite) supremum of those $t$ for which the
constraint holds. If $\tau^{\prime} < \tau$, 
the ${\mathcal H}$ compactness of $\{V(t): 0 \leq t \leq \tau^\prime\}$ 
implies that 
${\mathcal G}
(u, u_{t})(\tau^{\prime}) = 0$, 
otherwise, the constraint interval
may be extended within the maximal interval for the solution.
\end{proof}
\section{Concluding Remarks}
We have analyzed an operator forcing for the nonlinear 
damped wave equation (equivalently, telegraph equation),
which reflects recent studies of this equation, allowing for an
intermediate or concurrent physical process. 
The forcing is continuous and (locally) bounded
on $L^{2}$. 
It includes, but is not restricted to, function composition. As mentioned earlier, it
includes smoothing of quadratic nonlinearities, and may be capable of producing existence when combined with
compactness methods. This approach has been utilized in \cite{TDDFT} in a different context. For nonlinearities
defined by negative powers, new ideas are required in the case of homogeneous Dirichlet boundary
conditions, as considered here. However, in certain cases involving inhomogeneous boundary conditions, it is
expected that the approaches of this article apply to the case of negative powers in one dimension. 
We have defined maximal solutions for the nonlinear forcing,
obtained locally by successive approximation. 
The nonlinear theory is based upon a global linear weak solution theory,
and approximation of the forcing. An auxiliary result, not required for the existence, details  
a rigorous separation of
variables, associated with 
eigenfunction approximation;  
this reduces to familiar modal approximation in one dimension. 
Furthermore, the
semigroup is invariant on the individual eigenfunctions, 
which is maintained by the
convolution formula. 
The addition of a continuous constraint is compatible with the existence
of a maximal solution.

Uniqueness holds in the locally Lipschitz case. It is not clear
whether uniqueness holds more generally. If it is ultimately
demonstrated that uniqueness fails, then the model would require further
physical principles, likely from thermodynamics.
In the cited MEMS application \cite{LW2}, 
the forcing is the result of a complex physical
process involving elastic/electrostatic interactions, and is best
described via operator composition. Also, the wave motion tracked by
the model is constrained to avoid `touchdown'.
This constraint is
readily handled by the choice ${\mathcal G}(u) = 1 + u$ in one spatial
dimension. Our framework addresses this case. 

For a pedagogical introduction to the telegraph equation, cf.\ \cite{MAP}.
Its derivation dates to the 1880s, when it was derived by Heaviside to
describe attenuated electrical transmission. 

We conclude with some comparison of our hypotheses with those appearing in the literature. In
\cite{Mawhin,Mawhin2}, bounded solutions are derived. These results are more general than 
appears, since they are derived in combination with maximum principles. It appears that the maximum principles are
dimension dependent. Absent maximum principles, 
the restriction on the nonlinear forcing
appears to be somewhat strong. In \cite{Ber}, the forcing is assumed bounded and periodic solutions are derived via the
Leray-Schauder theorem. As previously mentioned, the forcing of this article is not restricted to function composition,
and may have greater applicability. Furthermore, by avoiding the use of classical fixed point theorems, and relying
on subsequential convergence, we are able to minimize the required hypotheses. 
\appendix
\section{Subsequential Convergence for Bounded Families}
In section 4, we applied two basic compactness results,
taken from
\cite{Caz} and \cite{Simon}. 
Here, we quote the underlying results for the reader's convenience. 
The first is cited from \cite[Proposition 1.1.2(i)]{Caz}.
\begin{proposition}[Cazenave] 
\label{B1}
Let $I$ be a bounded open interval
of ${\mathbb R}$, let $X \hookrightarrow Y$ be Banach spaces. 
Let $(f_{n})_{n \in {\mathbb N}}$
be a bounded sequence in $C({\bar I}; Y)$. Assume that  
$f_{n}(t) \in X \; \forall (n, t) \in {\mathbb N} \times I$ and that
$\sup \{\|f_{n}(t)\|_{X}, 
(n, t) \in {\mathbb N} \times I \} = K < \infty$. Assume further that $f_{n}$
is uniformly equicontinuous in $Y$. If $X$ is reflexive, then the
following holds.
There exists a function  
$f \in C({\bar I}; Y)$ which is weakly continuous ${\bar I} \mapsto X$ 
and a subsequence $n_{k}$ such that
\begin{equation*}
\forall t \in {\bar I}, \;  f_{n_{k}}(t) \rightharpoonup f(t), k
\rightarrow \infty, \; \mbox{in} \; X. 
\end{equation*}
\end{proposition}
It is not asserted that convergence is uniform in $t$.

The next result is cited from \cite[Theorem 2.3.14]{Simon}. It is a
generalized Arzela-Ascoli theorem.
\begin{proposition}[Simon]
\label{B2}
Let $X$ be a separable metric space and $Y$ a complete metric space, with
$C \subset Y$ compact. Let ${\mathcal F}_{0}$ be a family of uniformly
equicontinuous functions from $X$ to $Y$ with Range$(f) \subset C$
for every $f \in {\mathcal F}_{0}$. Then any sequence in ${\mathcal F}_{0}$ has a
subsequence converging at each $x \in X$. If $X$ is compact, then 
${\mathcal F}_{0}$ is precompact in the uniform topology.
\end{proposition} 

\end{document}